\documentclass[12pt]{article}
\usepackage[utf8]{inputenc}
\usepackage{amsbsy,amssymb,amscd,amsfonts,latexsym,amstext,amsmath,amsthm,graphicx,xcolor, mathtools,tikz}
\usepackage{hyperref}
\usepackage[margin=1 in]{geometry}

\newtheorem{thm}{Theorem}
\newtheorem{defn}{Definition}
\newtheorem{cor}{Corollary}
\newtheorem{fact}{Fact}
\newtheorem{lemma}{Lemma}
\newtheorem{quest}{Question}
\newcommand{\ecc}{\textrm{ecc}}
\newcommand{\dist}{\textrm{dist}}

\DeclarePairedDelimiter{\floor}{\lfloor}{\rfloor}

\usepackage{graphicx}

\usepackage{subfiles} 

\title{A Path Variant of the Explorer-Director Game on Graphs}
\author{Abigail Raz\footnote{Department of Mathematics, The Cooper Union, 10008 New York City, NY, USA, abigail.raz@cooper.edu - Corresponding author} \and Paddy Yang\footnote{Department of Electrical Engineering, The Cooper Union, 10008 New York City, NY, USA, paddy.yang@cooper.edu}}
\providecommand{\keywords}[1]{\textit{Keywords:} #1}
\begin{document}

\maketitle
\begin{abstract}
The Explorer-Director game, first introduced by Nedev and Muthukrishnan (2008), simulates a Mobile Agent exploring a ring network with an inconsistent global sense of direction. Two players, the Explorer and the Director, jointly control a token's movement on the vertices of a graph $G$ with initial location $v$. Each turn, the Explorer calls any valid distance, $d$, aiming to maximize the number of vertices the token visits, and the Director moves the token to any vertex distance $d$ away aiming to minimize the number of visited vertices. The game ends when no new vertices can be visited, assuming optimal play, and we denote the total number of visited vertices by $f_d(G,v)$. 
Here we study a variant where, if the token is on vertex $u$, the Explorer is allowed to select any valid \emph{path length}, $\ell$, and the Director now moves the token to any vertex $v$ such that $G$ contains a $uv$ path of length $\ell$. The corresponding parameter is $f_p(G,v)$. In this paper, we explore how far apart $f_d(G,v)$ and $f_p(G,v)$ can be, proving that for any $n$ there are graphs $G$ and $H$ with $f_p(G,v)-f_d(G,v)>n$ and $f_d(H,v)-f_p(H,v)>n$. 
\end{abstract}
\keywords{Games on graphs, distance, path length}

\section{Introduction}
Consider the following game played by two players - an Explorer and a Director. The game is played on a graph where players jointly move a single token around the vertices. The token originates at a given vertex and, in each round, the Explorer selects a distance for the token to move. The Director chooses to which vertex, of the appropriate distance from the current location, to move the token. The goal of the Explorer is to maximize the number of vertices the token visits, while the Director's goal is to minimize that same number. We traditionally consider the game to be over when the Director is able to indefinitely keep the token on an already visited vertex. We will denote the final number of visited vertices, assuming optimal play, when playing on graph $G$ with starting vertex $v$ by $f_d(G,v)$. Note that, by convention, we count the starting vertex as a visited vertex even if the token never returns there after the start of the game.

This game was first introduced by Nedev and Muthukrishnan as the Magnus-Derek Game to simulate the behavior of a \emph{Mobile Agent} exploring a ring network where there is an inconsistent global sense of direction \cite{NM}. To represent a ring network, Nedev and Muthukrishnan only focused on the Explorer-Director game on cycles. Note that on cycles, the starting vertex is irrelevant. Nedev and Muthukrishnan proved the following in 2008.
\begin{thm}\cite{NM}\label{NMthm}
For any vertex $v$ in $C_n$ we have
    \[f_d(C_n,v)=\begin{cases}
n & \text{if } n=2^k \text{ for some } k>0 \\
\frac{n(p-1)}{p} & \text{if } p \text{ is the smallest odd prime factor of } n.
\end{cases} \]

\end{thm}

The early research following \cite{NM} focused on alternative algorithms to achieve $f_d(C_n,v)$ more quickly with reduced computational complexity \cite{HPW, N10, N14, CLST}.
Additionally, Chen et al.\ developed strategies for other game variants, while Gerbner introduced an algebraic generalization for the original game, in which positions $g$ $\in$ $G$ are seen as elements of some cyclic group, $G$ \cite{CLST, G}.  Asgeirsson and Devlin further built on this idea \cite{AD}. Through relating the game to combinatorial group theory and the notion of \emph{twisted subgroup}, they expanded the scope of the game beyond cyclic groups to include other general finite groups, and named the game as the Explorer-Director Game. Finally, in \cite{ED} the first author alongside Devlin, Meger, and students through the 2020 Polymath REU investigated the Explorer-Director game on lattices and trees. Additionally, \cite{ED} reduces the study of $f_d (G,v)$ to the determination of the minimal sets of vertices that are \textit{closed} in a certain graph theoretic sense, providing a general lower bound for $f_d(G,v)$ in terms of the minimum eccentricity in $G$. Finally, in \cite{ED} a variant of the Explorer-Director game was introduced. In this variant, if the token is on vertex $u$, the Explorer is now allowed to select any valid \emph{path length}, $\ell$, and the Director can now move the token to any vertex $v$ such that $G$ contains a $uv$ path of length $\ell$.  This variation is perhaps a more natural extension of the original motivation of simulating the behavior of the Mobile Agent by allowing it to move arbitrarily since there is no guarantee the Mobile Agent moves along a shortest path. We will call this the \emph{path variant} of the Explorer-Director game and denote the final number of visited vertices, assuming optimal play, when playing on graph $G$ with starting vertex $v$ by $f_p(G,v)$. In this paper we focus on comparing how different $f_d(G,v)$ and $f_p(G,v)$ can be for certain nice graph families. 

The paper is organized as follows. In Section \ref{prelim}, we introduce necessary graph preliminaries. In Section \ref{bipan}, we show that for all bipanpositionable graphs on at least $4$ vertices, $f_p(G,v)=4$ for any starting vertex $v$. Then in Section \ref{HC}, we focus on a specific class of bipanpositionable graphs, namely hypercubes, and give a variety of bounds and exact results on $f_d(G,v)$ for various infinite families of hypercubes. Finally, in Section \ref{JF}, we introduce an infinite family of graphs where $f_p(G,v)$ grows arbitrarily larger than $f_d(G,v)$, with respect to the number of vertices.

\section{Preliminaries}\label{prelim}
We begin with a discussion of notation and terminology we will use throughout the paper. Much of what is discussed here is common, but we recall it here to avoid ambiguity. Given a graph $G$, let $V(G)$ denote the vertex set of $G$. We use $N(v)$ to denote the neighborhood of $v \in V(G)$. Recall that the distance between two vertices, $u$ and $v$, is denoted by $\dist(u,v)$, and is the length of a shortest $uv$ path. When $H$ is some subgraph of $G$ and $u,v \in V(H)$ we let $\dist_H(u,v)$ denote the length of the shortest $uv$ path in $H$. Furthermore, the \emph{eccentricity} of a vertex $v \in V(G)$, denoted $ecc_G(v)$, or simply $ecc(v)$, is the maximum distance from $v$ to any other vertex in the graph. 

In \cite{ED}, the authors gave the following definition of a \emph{closed} set to provide a combinatorial characterization of the structure of the set of visited vertices at the end of an optimally played Explorer-Director game. 

\begin{defn}\label{closeddef}\cite{ED}
On a graph $G=(V,E)$, we define a non-empty subset $U \subseteq V$ to be \emph{closed} when every vertex $u \in U$ has the property that for every vertex $v \in V$, there exists some vertex $x \in U$ such that $\dist(u,v) = \dist(u,x)$.
\end{defn}

This definition led to the following result from \cite{ED}.
\begin{thm}\label{closedthm}\cite{ED}
    Given a graph $G$ at the end of any Explorer-Director game on $G$, the set $U$ of all visited vertices always contains some closed set in $G$.
\end{thm}

This led to the following two general bounds. 

\begin{thm}\label{closedbound}\cite{ED}
For any graph $G$ we have $\min_{v \in V}f_d(G,v)$ is the cardinality of a minimum closed set in $G$.
\end{thm}

\begin{cor}\label{ecc}\cite{ED}
If in a game there exists a set of vertices $A \subseteq V$ which the Explorer can force the Director to visit at least once, then $\max_{v \in V(G)} f_d(G,v) \geq \min_{v \in A}\ecc(v) + 1$.
\end{cor}

As this paper focuses on comparing $f_d(G,v)$ and $f_p(G,v)$, we introduce an analogous concept of Definition \ref{closeddef} for the path variant.

\begin{defn}\label{pathcloseddef}
On a graph $G=(V,E)$, we define a non-empty subset $U \subseteq V$ to be \emph{path closed} when every vertex $u \in U$ has the property that for every positive integer $\ell$, if there is a vertex $v \in V(G)$ such that there is a $uv$ path of length $\ell$ then there must be an $x \in U$ such that there is a $ux$ path of length $\ell$.
\end{defn}

\section{Path Variant on Bipanpositionable Graphs}\label{bipan}
The notion of pancyclic graphs, graphs containing every possible cycle as a subgraph, was first introduced in 1971 by Bondy and has been widely studied since \cite{bondy}. Here we focus on a related definition for bipartite graphs. As bipartite graphs, by definition, do not contain odd cycles they cannot be pancyclic. However, a related notion of bipancyclicity was defined by Mitchem and Schmeichel in 1982 \cite{bipan}.
\begin{defn}\cite{bipan}\label{bipancycdef}
    We say a graph $G$ is \emph{bipancyclical} if $G$ contains every even cycle of length $4$ to $|V(G)|$ as a subgraph.
\end{defn}
Later, in 2006, Kao introduced the related notions of panpositionable and bipanpositionable graphs \cite{kao}. We will be focused only on bipartite graphs and thus only give the definition for bipanpositionable graphs below.
\begin{defn}\cite{kao}\label{bipanposdef}
    A hamiltonian bipartite graph $G$ is \emph{bipanpositionable} if for any two distinct vertices $x,y \in V(G)$ and for any integer $k$ with $\dist_G(x,y) \le k \le \frac{|V(G)|}{2}$ and $(k-d_G(x,y))$ even, there exists a hamiltonian cycle $C$ of $G$ such that $\dist_C(x,y)=k$.
\end{defn}
In \cite{kao} Kao proved that every hypercube $Q_n$ for $n \ge 2$ is bipanpositionable. In fact, Shih et. al. in \cite{shih} showed that hypercubes satisfy an even stronger property and are bipanpositionable bipancyclic. However, as we do not require this stronger property for Theorem \ref{path-hyper} we omit its definition here. 

\subsection{Results}
In this section, we explore the path parameter on hypercubes. While Theorem \ref{path-hyper} applies to all bipanpositionable  graphs it is worth noting that the primary bipanpositionable graphs in the literature are hypercubes.

\begin{thm}\label{path-hyper}
    If a graph $G$ is bipanpositionable with $4$ or more vertices then $f_p(G,v)$ = $4$ for every $v \in V(G)$.
\end{thm}
\begin{proof}
   Let $G$ be a bipanpositionable graph with at least four vertices and suppose the starting vertex of the game is $a$. First we will show there exists a $C_4$ subgraph of $G$ containing $a$. Note that the only bipartite hamiltonian graph on four vertices is $C_4$, so we may now assume $|V(G)|\ge 6$. Let $b$ be an arbitrary neighbor of $a$. Since $G$ is bipanpositionable with $6$ or more vertices, we know there is a hamiltonian cycle $C$ of $G$ such that $\dist_C(a,b)=3$. Taking the length $3$ $ab$ path on $C$ together with the $ab$ edge in $G$ forms a $C_4$ containing $a$. In this $C_4$ subgraph let $b$ and $c$ be the two neighbors of $a$ and $d$ the remaining vertex. Note that as $G$ is bipartite, a necessary condition to being bipanpositionable, we know that neither the edge $ad$ nor the edge $bc$ can exist. Thus $\dist_G(a,d)=\dist_G(b,c)=2$. We claim that in playing the path variant of the Explorer-Director game the director can indefinitely keep the token on $W=\{a,b,c,d\}$.
   
 Since $W$ induces a $C_4$ in $G$ we clearly know that if the Director calls a path length of $1$ or $2$ the Director can simply move the token from any vertex in $W$ to another vertex in $W$. We will now show that for any path length $\ell$ from $3$ to $|V(G)|-1$ if the token is on vertex $a$ the Director can ensure the token moves to another vertex in $W$ after the Explorer calls $\ell$. This argument will hold identically if the token is on $b, c, $ or $d$ thus proving that $f_p(G,v) \le 4$ for any $v\in V(G)$.

 First consider the case where $\ell$ is even. The largest possible even value of $l$ that the Explorer could choose is $|V(G)| - 2$, as $|V(G)|$ must be even since $G$ is a hamiltonian bipartite graph. Note since $G$ is bipanpositionable we know that for any even $\ell$ with $2\le \ell \le \frac{|V(G)|}{2}$ there is a hamiltonian cycle, $C$, where $\dist_C(a,d)=\ell$. We then have paths between $a$ and $d$ along $C$ of both length $\ell$ and $|V(G)|-\ell$. Thus we can clearly find a path of any even length between $a$ and $d$.

 In the case where $\ell$ is odd we similarly have, by the fact that $G$ is bipanpositionable, that for any odd $\ell$ between $1$ and $\frac{|V(G)|}{2}$ there is a hamiltonian cycle $C$ where $\dist_C(a,b)=\ell$. Again this implies that we can find paths between $a$ and $b$ along $C$ of both length $\ell$ and $|V(G)|-\ell$. Thus we can clearly find a path of any odd length between $a$ and $c$. Thus we have proved that $W$ is a path closed set, implying $f_p(G) \le 4$. 

 For the lower bound we note that the Director cannot restrict the token to a set with less than $4$ vertices. As $G$ is bipartite the smallest possible cycle length is $4$ and thus if the Explorer calls a path length of 1 or 2 the Director's options, under the rules of the path variant, are identical to those under the distance variant. Note that if the Explorer begins by calling a path length of $1$ the token must move to a vertex $u$ adjacent to $a$. If the Explorer next calls a path length of $2$ the token must again move to yet another vertex $v$. Now we simply must show that the Explorer can force the token to another vertex. Let the Explorer call a path length of $1$ again. The token cannot move back to $u$ as $\dist(u,v)=2$. If the Director moves the token to a vertex other than $a$ we are done, so we may assume the token returns to $a$. This implies that $\dist(a,v)=1$ and thus there is not a path of length $2$ between $a$ and $u$ or $a$ and $v$. Therefore, if the Explorer calls $2$ once again the token will need to visit a fourth vertex as desired. 

    Since both the upper and lower bound collapse onto $4$, then we can conclude $f_p(G,v)=4$ for any bipanpositionable graph and any starting vertex $v$.
 
\end{proof}

\begin{cor} For any hypercube $Q_n$ and any starting vertex $v$ we have
    $f_p(Q_n,v) =
        \begin{cases} 
        1 & \text{if } n = 0 \\
        2 & \text{if } n = 1 \\
        4 & \text{if } n \geq 2.\\
        \end{cases}$
\end{cor}

\begin{proof}
    If $n$ is equal to $0$, then there is only one vertex in the graph. Thus, $f_p(Q_0,v)=1$. Similarly $Q_1 \cong K_2$, so clearly $f_p(Q_1,v)=2$. 
As proved in \cite{kao} hypercubes where $n$ is at least $2$ are bipanpositionable. Therefore, Theorem \ref{path-hyper} immediately implies that $f_p(Q_n,v)=4$ for $n \geq 2$. 
\end{proof}

\section{Distance Variant on Hypercubes}\label{HC}
In this section we use the standard representation of the $2^n$ vertices in $Q_n$ via unique binary strings of length $n$. Recall that for $u, v \in V(Q_n)$ an edge $uv$ in $Q_n$ exists exactly when $u$ and $v$ have exactly one bit which differs in their binary string notation. If $v\in V(Q_n)$ we will let $v_i$ denote the $i$th bit of $v$ in its binary representation. We will let the \emph{weight} of a vertex refer to the number of 1's in the binary representation. 

It is not hard to check that hypercubes are vertex-transitive. Thus the starting vertex in the Explorer-Director game is irrelevant for hypercubes, i.e. $f_d(Q_n, v)=f_d(Q_n,u)$ for any two vertices $u$ and $v$ in $V(Q_n)$. Therefore, for the remainder of this section we will just write $f_d(Q_n)$ for $f_d(Q_n,v)$ and assume the starting vertex is always $\overline{0}$, the vertex of weight $0$. Thus Theorem \ref{closedbound} immediately implies that $f_d(Q_n)$ is equal to the cardinality of a minimum closed set in $G$.
 Note that the eccentricity of every vertex in $Q_n$ is simply $n$, and for every vertex $v$ there is a unique corresponding vertex $u$ of distance $n$ from $v$. 
Thus Corollary \ref{ecc} immediately implies that $f_d(Q_n) \ge n+1$. 

Note that the distance between any two vertices in a hypercube is simply the hamming distance between their binary representations, i.e. the number of locations in which their bits differ. Thus it is not too hard to check the following fact, but we will include a proof for completeness. 
\begin{fact}\label{sum-even}
    Given vertices $u,v,w \in V(Q_n)$ we have $\dist(u,v) + \dist(v,w) + \dist(w,u)$ must be even.
\end{fact}
\begin{proof}
We will consider the distances one bit at a time. For a given bit suppose all three vertices have the same value. Then that bit contributes 0 to the total distance sum. Now, suppose two of the vertices have the same value and the other does not. Without loss of generality, say the two vertices with the same digit value are $u$ and $v$ while $w$ has a different digit in that bit. In this case, this bit will contribute 1 to each of $\dist(w,u)$ and $\dist(v,w)$ and 0 to $\dist(u,v)$. Note that as there are only two options for each bit these are the only two possible scenarios. Thus each bit contributes an even number, either zero or two, to $\dist(u,v) + \dist(v,w) + \dist(w,u)$, so $\dist(u,v) + \dist(v,w) + \dist(w,u)$ must be even.
\end{proof}

Furthermore, note that the fact that each vertex has a unique vertex  distance $n$ away implies the following. 

\begin{lemma}\label{fd-even}
    For any hypercube $Q_n$ we have that $f_d(Q_n)$ is always even.
\end{lemma}

\begin{proof}
Recall that by Theorem \ref{closedbound} we know $f_d(Q_n)$ is the cardinality of a minimum closed set in $Q_n$. Since all vertices in $Q_n$ have a unique vertex of distance $n$ away we claim that every closed set of $Q_n$ must have even cardinality. Specifically, if $W$ was a closed set in $Q_n$ such that $|W|$ was odd then there would be at least one vertex $w \in W$ such that there is no vertex in $W$ of distance $n$ from $w$. Thus $W$ is not a closed set, and we know $f_d(Q_n)$ is always even.
\end{proof}

We are able to use Lemma \ref{fd-even} to prove the following two lemmas which each give upper bounds on $f_d(Q_n)$. Note that when $n$ is not a power of 2 the upper bounds in Lemmas \ref{H-UB1} and \ref{H-UB2} are useful in different scenarios.
\begin{lemma}\label{H-UB1}
    If $n$ is not a power of $2$, then 
    \begin{equation*}
        f_d(Q_n) \leq \left(\frac{p-1}{p}\right)2n, \\ \text{ where p is the smallest odd prime factor of n.}
    \end{equation*}
    If $n$ is a power of 2, we have
    \begin{equation*}
        f_d(Q_{n}) \leq 2n.
    \end{equation*}
\end{lemma}
\begin{proof}
    Given any hypercube $Q_n$, it is possible to create a cycle with $2n$ vertices such that between two vertices, $u$ and $v$, the distance on the graph is equal to the distance on the cycle. For instance, take the cycle formed by $W \cup U$ where
    \begin{align*}
       W &= \{v \in V(Q_n): \text{for some }0 \le j \le n \text{ we have } v_i = 0 \text{ for }i\le j\text{ and } v_i = 1 \text{ else}\} \\
       U &= \{v \in V(Q_n): \text{for some }0 \le j \le n \text{ we have } v_i = 1 \text{ for }i\le j\text{ and } v_i = 0 \text{ else}\}.
    \end{align*}
 For example, in $Q_3$ we have $W\cup U=\{000,001,011,111,110,100\}$. Due to the manner in which this cycle is constructed, the distance between vertices in the cycle is equal to the distance between the vertices in the graph. The vertices in this cycle is form a closed set since for any possible distance on the hypercube, there exists another vertex on the cycle of that distance away. Since the Director can restrict the token to a subgraph which is isomorphic to $C_{2n}$ our result follows immediately from Theorem \ref{NMthm}.
\end{proof}
We can further generalize the powers of 2 bound from Lemma \ref{H-UB1}. 
\begin{lemma}\label{H-UB2}
  For any hypercube $Q_n$ and integer $x$ satisfying $\floor{\frac{n}{2}} \leq x < n$ we have $f_d(Q_n) \leq 2f_d(Q_x)$. 
\end{lemma}

\begin{proof}
    Let $S$ be a closed set of $Q_{x}$. We will let $0_i$ and $1_i$ denote the binary string of length i with all 0's or all 1's, respectively. Then, consider the following sets of vertices in $Q_n$.
    \begin{itemize}
        \item $A=\{(0_{n-x},s): s \in S\}$ 
        \item $B=\{(1_{n-x},s): s \in S\}$ 
    \end{itemize}
    We claim $C= A \cup B$ is a closed set of $Q_n$. Suppose the token is on any vertex in $C$. Without loss of generality, we will say the token is on $(0_{n-x},s)$ for some $s \in S$. If the Explorer selects any distance from $1$ to $x$ inclusively, we can guarantee there is another vertex in set $A$ of the given distance since for all vertices in set $S$, there exist another vertex of any distance $1$ to $x$ away in $S$. If the Explorer selects a distance, $d$, greater than $x$, we can first move a distance of $n-x$ from $(0_{n-x},s)$ to $(1_{n-x},s)$. Then as $\floor{\frac{n}{2}} \leq x < n$ we know $1 \le d-(n-x) \le x$. From $(1_{n-x},s)$, there exists a  vertex of distance $d-(n-x)$ away from it in set $B$. Since the token moved the by changing the first $n-x$ digits to move between set $A$ and $B$, and we change the latter $n$ digits when moving within set $B$, the total distance between the initial and final vertex is equal to the distance it takes to switch sets plus the distance traveled in the $B$ as desired.
\end{proof}

Next we will show three lower bounds which hold for various families of hypercubes. For some specific families of hypercubes these lower bounds combined with the upper bounds above yield exact results. 

\begin{lemma}\label{1mod4-LB}
    For all integers $k \geq 1$ we have $f_d(Q_{4k+1}) \geq 4k+4$.
\end{lemma}
\begin{proof}
Again Theorem \ref{closedbound} gives that $f_d(Q_{4k+1})$ must be the cardinality of a minimum closed set in $Q_{4k+1}$. Recall that Lemma \ref{fd-even} states that $f_d(Q_{4k+1})$ must be even. Thus, it only remains to show that there is no closed set in $Q_{4k+1}$ of size at most $4k+2$. Let $W$ be a closed set of $Q_{4k+1}$. Without loss of generality, assume $\overline{0}\in W$. Thus $W$ must also contain a vertex of distance $i$ from $\overline{0}$ for every $1\le i \le 4k+1$. We will identify one vertex of each distance from $\overline{0}$ in $W$ and label them $w^i$ where $i$ denotes $\dist(\overline{0},w^i)$, i.e. the number of ones in the binary representation of $w^i$. Recall we also refer to the number of ones in the binary representation of a vertex as the weight of that vertex. Thus, we know $|W| \ge 4k+2$.

We will assume, for the sake of contradiction, that $|W|=4k+2$. This immediately implies that for every $w \in W$ and every distance $1\le i\le 4k+1$ there is a unique vertex $v \in W$ where $\dist(w,v)=i$. We will show that this is, in fact, impossible. Note that for any $w^x\in W$ of even weight $x$, we have $\dist(w^2, w^x)\in \{x-2, x, x+2\}$. We first consider $\dist(w^2,w^4)$. We know $\dist(\overline{0},w^2)=2$ and $\dist(\overline{0},w^4)=4$ so $w^2$ cannot have another vertex in $W$ of distance $2$ and $w^4$ cannot have another vertex in $W$ of distance $4$. Thus, we must have $\dist(w^2, w^4)=6$. Note that if $k=1$ then this immediately gives a contradiction as every pair of vertices in $Q_5$ are distance at most $5$ apart, and thus either $w^2$ has two vertices of distance 2 away in $W$ or $w^4$ has two vertices of distance $4$ away in $W$, as desired.

Now assume $k \ge 2$. We claim that for all $1\le \ell <k$ that $\dist(w^2, w^{4l})=4\ell+2$. We will prove this statement by induction. We showed in the prior paragraph that for $k \ge 2$ we must have $\dist(w^2, w^4)=6$. Assume for induction that for some $l$ satisfying $1\le \ell <k$ we have $\dist(w^2, w^{4\ell})=4\ell+2$. Now we consider $\dist(w^2, w^{4\ell+4})$. As above $\dist(w^2, w^{4\ell+4})\in \{4\ell+2, 4\ell+4, 4\ell+6\}$. However, we know $\dist(\overline{0}, w^{4\ell+4})=4\ell+4$ and, by our inductive hypothesis, that $\dist(w^2, w^{4\ell})=4\ell+2$. Thus in order to avoid a repeated distance for both $w^2$  and $w^{4l+4}$ in $W$ we must have $\dist(w^2, w^{4\ell+4})=4\ell+6$, as desired. 

Now we consider $w^{4k}$. Note that as every pair of vertices in $Q_{4k+1}$ has distance at most $4k+1$, we know $\dist(w^2, w^{4k}) \neq 4k+2$. Therefore $\dist(w^2, w^{4k})\in \{4k-2, 4k\}$. However we know $\dist(\overline{0},w^{4k})=4k$ and $\dist(w^{2},w^{4k-4}) =4k-2$. Therefore, there must be a vertex $w \in W$ and an even distance $i$ such that there are two vertices in $W$ of distance $i$ from $w$. This contradicts our assumption that $f_d(Q_{4k+1})=4k+2$, as desired.
\end{proof}

We now consider a similar lower bound for hypercubes of dimension exactly a multiple of 4.

\begin{lemma}\label{0mod4-LB}
    $f_d(Q_{4k}) \geq 4k+4$ for all integers $k \geq 1$.
\end{lemma}
\begin{proof}
    As in the proof of Lemma \ref{1mod4-LB} we know that $f_d(Q_{4k})$ must be the cardinality of a minimum closed set in $Q_{4k}$ and that $f_d(Q_{4k})$ must be even. Thus it only remains to show that there is no closed set in $Q_{4k}$ of size at most $4k+2$. Again let $W$ be a minimum closed set in $Q_{4k}$, and without loss of generality assume $\overline{0}\in W$. Again, we will identify one vertex of each distance from $\overline{0}$ in $W$ and label them $w^i$ where $i$ denotes $\dist(\overline{0},w^i)$. Note that in the case of $w^{2k}$ the unique vertex, $v$, of distance $4k$ from $w^{2k}$ is also weight $2k$, so $W$ must contain at least two weight $2k$ vertices. 

    We know, by Fact \ref{sum-even}, that given any three vertices, the sum of the distances between them must be even. Suppose we only consider the odd weight vertices $\{w^1,w^3,w^5,w^7,\ldots,w^{4k-1}\}$. Since they are all of an odd distance from vertex $\overline{0}$, then the distance between any two odd weight vertices must be even. Similarly, we know that if $i$ is odd and $j$ is even then $\dist(w^i,w^j)$ must be odd. Note that $|\{w^1,w^3,w^5,w^7,\ldots,w^{4k-1}\}|=2k$. Thus, there must exist some even value $x$ between $2$ and $4k$ such that there is no $w^i$ with $\dist(w^1,w^i)=x$. Since $W$ is closed, it must contain a vertex $a$ which is distance $x$ from $w^1$. Thus, $W$ must contain at least $\{\overline{0},v,a\} \cup \{w^i: 1 \le i \le 4k\}$, which has cardinality $4k+3$ as desired.

\end{proof}

Finally, we will discuss an exact result for $f_d(Q_n)$ when $n$ is slightly less than a power of two. 

\begin{thm}
    $f_d(Q_{2^k-x}) = 2^k$ for $x\in \{1,2,3,4\}$, $k \ge 1$, and $2^k-x\ge 2^{k-1}$. 
\end{thm}

\begin{proof}
First, we will prove the lower bounds. When $x=1$, Corollary \ref{ecc} immediately provides the lower bound. When $x=2$ we note the lower bound given by Corollary \ref{ecc} is $2^{k}-1$, an odd number, and thus, combining this with Lemma \ref{fd-even} immediately gives the desired lower bound. In the cases of $x=3$ and $x=4$, we note that $2^k-3 \equiv 1 \mod 4$ and $2^k-4 \equiv 0 \mod 4$. Thus, Lemmas \ref{1mod4-LB} and \ref{0mod4-LB} give the desired lower bounds in these cases. 

We will prove the upper bounds sequentially. For $x=1$, we use induction along with Lemma \ref{H-UB1}. Specifically, if $k=1$, then we simply have $Q_{2^1-1}\cong K_2$ and thus $f_d(Q_1)=2^1$ as desired. We now assume, for induction, that for some $k \ge 1$ we have $f_d(Q_{2^k-1}) = 2^k$. Note that $\floor{\frac{2*2^{k}-1}{2}} = 2^k-1$. Thus by Lemma \ref{H-UB2}, we know $f_d(Q_{2^{k+1}-1}) \le 2f_d(Q_{2^k-1})$. By our inductive hypothesis, $f_d(Q_{2^k-1}) =2^k$ giving $f_d(Q_{2^{k+1}-1}) \le 2^{k+1}$ as desired.
  Since the upper and lower bounds for $f_d(Q_{2^k-1})$ are both $2^k$, we can conclude $f_d(Q_{2^k-1})=2^k$. 

  In the case of $x=2$, we can use our result for $x=1$ noting that for $k \ge 2$ we have $\frac{2^k-2}{2}$ equals $2^{k-1}-1$ where $k-1$ is at least $1$. Again using Lemma \ref{H-UB2}, we now have $f_d(Q_{2^k-2}) \le 2 f_d(Q_{2^{k-1}-1})=2^{k}$ by our previous paragraph. Thus, we also conclude $f_d(Q_{2^k-2})=2^k$. 

  Similar arguments hold for both $x=3$ and $x=4$. Here we have $\floor{\frac{2^k-x}{2}} = 2^{k-1}-2$, so again using Lemma \ref{H-UB2}, we have $f_d(Q_{2^k-x}) \le 2 f_d(Q_{2^{k-1}-2})=2^{k}$, by our previous paragraph. Again the lower and upper bounds collapse showing that $f_d(Q_{2^k-x})$ is always $2^k$ when $x \in \{1,2,3,4\}$.

\end{proof}

Numerical computations have also provided some additional results for small values of $n$. Specifically, we know:
\begin{center}
  $f_d(Q_n) =
        \begin{cases} 
        1 & \text{if } n = 0 \\
        2 & \text{if } n = 1 \\
        4 & \text{if } 2 \le n\le 3\\
        8 & \text{if } 4 \le n\le 7\\
        16 & \text{if } n=8\\
         12 & \text{if } n=9.\\
        \end{cases}$
        \end{center}
The final case of $f_d(Q_9)$ shows that $f_d(Q_n)$ is not always a power of two, and, more interestingly, that $f_d(Q_n)$ is not weakly increasing in $n$. Unfortunately, our bounds and small numerical cases did not yield a single likely conjecture on a closed form, in terms of $n$, for $f_d(Q_n)$.
\section{Graphs for which $f_d(G) < f_p(G)$}\label{JF}
As noted in the introduction, the traditional Explorer-Director parameter and the path variant are identical in the cases of cycles, the original motivation. When extending to other families of graphs, it is natural to consider the path variant as the original motivation was derived from considering the behavior of a mobile agent with an inconsistent sense of direction \cite{NM}. In general graphs, this inconsistent sense of direction may lead the agent, or token, to traverse arbitrary paths rather than only geodesics. Additionally, for the graphs explored in \cite{ED}, we note that path parameter is either identical to the classical distance parameter, as in the case of trees, or significantly smaller, as in the case of rectangular lattices. Furthermore, in other graph families, such as hypercubes, we are able to show that the path parameter was constant, relative to the number of vertices, while the distance parameter grew with the number of vertices. In these cases, the ability for the Director to choose any path, rather than just geodesics, far outstrips the power of the Explorer to call longer path lengths leading to a much smaller parameter. 

It is then natural to consider if in the path variant the Director always has more power, leading to a smaller parameter, or if there are instances where the $f_d(G,v)< f_p(G,v)$. In this section, we present one such infinite family of graphs. 
\subsection{Construction}
We will denote our infinite family of graphs satisfying $f_d(G,v)< f_p(G,v)$ by $CF_n$ and refer to them as cuttlefish graphs. $CF_n$ contains a single cycle of length $n$ with vertices around the cycle as denoted by $u_i$ for $1\le i \le n$ and edges $u_iu_{i+1 \mod n}$. We then append two paths of length $\floor{\frac{n}{2}-1}$ to $u_1$ and $u_2$. We will denote the new vertices on the path appended to $u_1$ by $v_i$ for $1\le i \le \floor{\frac{n}{2}}-1$ where $u_1v_1$ and $v_iv_{i+1}$, for $1\le i \le \floor{\frac{n}{2}}-2$ are all edges. Similarly, we will denote the new vertices on the path appended to $u_2$ by $w_i$ for $1\le i \le \floor{\frac{n}{2}}-1$ where $u_2w_1$ and $w_iw_{i+1}$, for $1\le i \le \floor{\frac{n}{2}}-2$, are all edges. For ease of notation let $u_1=v_0$ and $u_2=w_0$. Note that $CF_n$ will have $n+ 2(\floor{\frac{n}{2}}-1)$ vertices and edges.
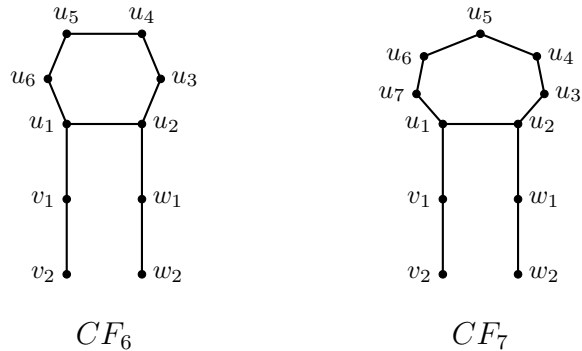
\begin{figure}
\begin{center}
\begin{tikzpicture}
\draw [fill=black] (0,0)circle [radius=0.05] node[text=black,anchor=east]{\footnotesize $u_1$};
\draw [fill=black] (0,1.2)circle [radius=0.05] node[text=black,anchor=south]{\footnotesize $u_5$};
\draw [fill=black] (1,1.2)circle [radius=0.05] node[text=black,anchor=south]{\footnotesize $u_4$};
\draw [fill=black] (1,0)circle [radius=0.05] node[text=black,anchor=west]{\footnotesize $u_2$};
\draw [fill=black] (-.25,.6)circle [radius=0.05] node[text=black,anchor=east]{\footnotesize $u_6$};
\draw [fill=black] (1.25,.6)circle [radius=0.05] node[text=black,anchor=west]{\footnotesize $u_3$};
\draw [black, thick] (0,0)--(-.25,.6);
\draw [black, thick] (-.25,.6)--(0,1.2);
\draw [black, thick] (1,1.2)--(0,1.2);
\draw [black, thick] (1.25,.6)--(1,1.2);
\draw [black, thick] (1.25,.6)--(1,0);
\draw [black, thick] (0,0)--(1,0);
{\draw [fill=black] (0,-1)circle [radius=0.05] node[text=black,anchor=east]{\footnotesize $v_1$};}
{\draw [fill=black] (0,-2)circle [radius=0.05] node[text=black,anchor=east]{\footnotesize $v_2$};}
{\draw [fill=black] (1,-1)circle [radius=0.05] node[text=black,anchor=west]{\footnotesize $w_1$};}
{\draw [fill=black] (1, -2)circle [radius=0.05] node[text=black,anchor=west]{\footnotesize $w_2$};}
{\draw [black, thick] (0,0)--(0,-1);}
{\draw [black, thick] (0,-1)--(0,-2);}
{\draw [black, thick] (1,0)--(1,-1);}
{\draw [black, thick] (1,-1)--(1,-2);}
\filldraw[white] (.5,-2.5) circle[radius=.01] node[text=black, anchor=north]{$CF_6$};

\draw [fill=black] (5,0)circle [radius=0.05] node[text=black,anchor=east]{\footnotesize $u_1$};
\draw [fill=black] (4.65,.4)circle [radius=0.05] node[text=black,anchor=east]{\footnotesize $u_7$};
\draw [fill=black] (5.5,1.2)circle [radius=0.05] node[text=black,anchor=south]{\footnotesize $u_5$};
\draw [fill=black] (6.35,.4)circle [radius=0.05] node[text=black,anchor=west]{\footnotesize $u_3$};
\draw [fill=black] (6,0)circle [radius=0.05] node[text=black,anchor=west]{\footnotesize $u_2$};
\draw [fill=black] (6.25,.9)circle [radius=0.05] node[text=black,anchor=west]{\footnotesize $u_4$};
\draw [fill=black] (4.75,.9)circle [radius=0.05] node[text=black,anchor=east]{\footnotesize $u_6$};

\draw [black, thick] (5,0)--(4.65,.4);
\draw [black, thick] (4.65,.4)--(4.75,.9);
\draw [black, thick] (4.75,.9)--(5.5,1.2);
\draw [black, thick] (6.25,.9)--(5.5,1.2);
\draw [black, thick] (6.25,.9)--(6.35, .4);
\draw [black, thick] (6,0)--(6.35, .4);
\draw [black, thick] (5,0)--(6,0);
{\draw [fill=black] (5,-1)circle [radius=0.05] node[text=black,anchor=east]{\footnotesize $v_1$};}
{\draw [fill=black] (5,-2)circle [radius=0.05] node[text=black,anchor=east]{\footnotesize $v_2$};}
{\draw [fill=black] (6,-1)circle [radius=0.05] node[text=black,anchor=west]{\footnotesize $w_1$};}
{\draw [fill=black] (6, -2)circle [radius=0.05] node[text=black,anchor=west]{\footnotesize $w_2$};}
{\draw [black, thick] (5,0)--(5,-1);}
{\draw [black, thick] (5,-1)--(5,-2);}
{\draw [black, thick] (6,0)--(6,-1);}
{\draw [black, thick] (6,-1)--(6,-2);}
\filldraw[white] (5.5,-2.5) circle[radius=.01] node[text=black, anchor=north]{$CF_7$};
\end{tikzpicture}
\end{center}
\caption{Small examples of cuttlefish graphs}
\end{figure}
\subsection{Results}
We will show that this family of graphs does have the desired property that $f_d(CF_n,v)< f_p(CF_n,v)$ for every choice of starting vertex. More precisely we show the following. 
\begin{thm}\label{JF-path}
    For any $n\geq 10$ and any starting vertex $v \in V(CF_n)$ 
    \begin{align}
        f_p(CF_n,v) = 3\left\lfloor\frac{n}{2}\right\rfloor && \text{when $n$ is odd } \\
        f_p(CF_n,v)= \frac{3n}{2}-1 && \text{when $n$ is even}.
    \end{align}

\end{thm}

Furthermore, in the case of the classical distance game we are able to show the following. 
\begin{thm}\label{JF-dist}
For any even $n\geq 5$ and any starting vertex $v \in V(CF_n)$
\begin{equation}
    f_d(CF_n,v) = 2\left\lfloor\frac{n}{2}\right\rfloor.
 \end{equation}
Additionally for any odd $n \geq 5$ 
\begin{align}
     f_d(CF_n,v) &= 2\left\lfloor\frac{n}{2}\right\rfloor && \text{when $v \in V(CF_n) \setminus \{u_{\frac{n+3}{2}}\}$}\\
     f_d(CF_n,u_{\frac{n+3}{2}}) &= 2\left\lfloor\frac{n}{2}\right\rfloor+1.
\end{align}
\end{thm}
Together Theorems \ref{JF-path} and \ref{JF-dist} show that for any fixed $k$, we can find a graph for which $f_p(G,v)-f_d(G,v)>k$ for some vertex $v$. However, the following remains open. 
\begin{quest}
    Given any $k$, is it possible to find a graph $G$ and vertex $v \in V(G)$ such that $
    \frac{f_p(G,v)}{f_d(G,v)}>k$?
\end{quest}

First we will prove Theorem \ref{JF-path}. 
\begin{proof}
    We will start by proving that the values given in Theorem \ref{JF-path} are lower bounds. Specifically, we will give a strategy for the Explorer that will force the token to visit at least the required number of vertices regardless of the Director's choices. We will first show that regardless of the starting vertex $v$, the Explorer can always force the token onto one of the leaves of $CF_n$. In almost every case, this only requires a single move. The only exception is when $n$ is odd and $v = u_{\frac{n+3}{2}}$ which is the vertex directly opposite $u_1$ and $u_2$ on the cycle. In this case, every path with $u_{\frac{n+3}{2}}$ as an endpoint has length at most $n-1$. Thus, if the token begins at $u_{\frac{n+3}{2}}$, the Director can keep the token on the cycle after the first round. However, in this case, the Explorer can simply choose a path length of $1$ in the first round to move the token to a vertex other than $u_{\frac{n+3}{2}}$. Once the token is on some vertex $a \neq u_{\frac{n+3}{2}}$, we note that the unique maximum length path with $a$ as an endpoint also has one of the two leaves as an endpoint. Thus, we may assume the token begins, without loss of generality, on $v_{\floor{\frac{n}{2}}-1}$.  

    Again as we are beginning with the proof of the lower bound in this section we provide a strategy for the Explorer which will guarantee that at least the desired number of vertices are visited. The Explorer's strategy will be as follows. If the token is on $v_i$ where $1 \le i\le  \floor{\frac{n}{2}}-1$, then let 
    $$U = \{v_j: j<i \text{ and } v_j \text{ has not yet been visited}\}.$$ The Explorer will select a path length of $d=\min\{\dist(v_i, v_j): v_j\in U\}$. Let $v_f \in U$ be the vertex such that $d=\dist(v_i, v_f)$. Now the Director has two choices, either they can move the token to $v_f$ or, if possible, move the token ``backwards" to $v_{i+d}$. However, once $i+d> \floor{\frac{n}{2}}-1$, the Director will have no choice but to move the token to $v_f$. 
    
    Ultimately, this strategy will cause each vertex on $\{v_i\}_{i=0}^{ \floor{\frac{n}{2}}-1}$ to be visited before the token can leave $\{v_i\}_{i=0}^{ \floor{\frac{n}{2}}-1}$. Furthermore, the last vertex in $\{v_i\}_{i=0}^{ \floor{\frac{n}{2}}-1}$ to be visited under this strategy must be $v_0=u_1$. Once the token is at $u_1$, the Explorer can then call $n+\left\lfloor\frac{n}{2}\right\rfloor-2$ to force the token to $w_{\floor{\frac{n}{2}}-1}$. The Explorer can then repeat the above strategy, now on the other path to force all the vertices $\{w_i\}_{i=0}^{\floor{\frac{n}{2}}-1}$ to be visited, ending at $w_0=u_2$. At this point, a total of $2\left\lfloor \frac{n}{2}\right\rfloor$ vertices have been visited by the token after the token first arrived at $v_{\floor{\frac{n}{2}}-1}$ - recall the token may have visited up to two additional vertices prior to this point depending on the starting vertex and parity of $n$.

    At this point in the process, the token must be at $u_2$. We will now split into cases depending on the parity of $n$. We will first assume that $n$ is odd. The Explorer will next select a path length of $\left\lfloor \frac{n}{2}\right\rfloor +1$. Note that there is no path of this length from $u_2$ to any $v_i$ or $w_i$. In fact, the only vertices $a$ for which there is a $u_2a$ path of the desired length are $u_{\floor{\frac{n}{2}}+3}$ and $u_{\floor{\frac{n}{2}}+2}$. Note that for $a \in \{u_{\floor{\frac{n}{2}}+3},u_{\floor{\frac{n}{2}}+2}, u_{\floor{\frac{n}{2}}+1}\}$ and $d \in\{ 1, 2, 3, \ldots, \left\lfloor\frac{n}{2}\right\rfloor-2\}$, we have every $av_i$ path and $aw_i$ path is of length at least $\left\lfloor\frac{n}{2}\right\rfloor-1$, but there is some $u_i$ such that there is an $au_i$ path of length $d$. Furthermore, if we have $d_1, d_2\in\{ 1, 2, 3, \ldots,\left\lfloor\frac{n}{2}\right\rfloor-2\}$, then we know it is impossible to have some $u_i$ such that there is both an $au_i$ path of length $d_1$ and an $au_i$ path of length $d_2$. Furthermore, if the token is on either of the two leaves, the Explorer can then call a path length of $n-1$ to force the token to return to a vertex in $\{u_{\floor{\frac{n}{2}}+3},u_{\floor{\frac{n}{2}}+2}, u_{\floor{\frac{n}{2}}+1}\}$. This implies that the Explorer can force the token to return to $\{u_{\floor{\frac{n}{2}}+3},u_{\floor{\frac{n}{2}}+2}, u_{\floor{\frac{n}{2}}+1}\}$ an arbitrarily large number of times. 
    
It just remains to show that the Director cannot restrict the token to some set of the form $A \cup \{v_i\}_{i=0}^{\lfloor \frac{n}{2}\rfloor - 1} \cup \{w_i\}_{i=0}^{\lfloor \frac{n}{2}\rfloor - 1}$ where $A \subseteq \{u_i\}_{i=3}^n$ and $|A|\le \lfloor\frac{n}{2}\rfloor -1$. 
To show this first note that if the token visits $u_{\floor{\frac{n}{2}}+2}$ at least $\left\lfloor\frac{n}{2}\right\rfloor$ times, then we can guarantee at least $\left\lfloor\frac{n}{2}\right\rfloor$ vertices in $\{u_i\}_{i=3}^n$ must be visited. This is because if the Explorer selects a path length of $i$ when the token is on $u_{\floor{\frac{n}{2}}+2}$ for the $i$th time, then for $1 \le i \le \lfloor \frac{n}{2}\rfloor -1$ the Director will be forced to move the token to a vertex which has never been visited immediately following the visit to $u_{\floor{\frac{n}{2}}+2}$.
If this is not the case, then the token must return to both $u_{\floor{\frac{n}{2}}+3}$ and $ u_{\floor{\frac{n}{2}}+1}$ arbitrarily many times, as we can return to both leaves arbitrarily many times. Note that for $A$ as above we can only afford to have $\lfloor\frac{n}{2}\rfloor -3$ vertices in $A$ beyond $u_{\floor{\frac{n}{2}}+3}$ and $ u_{\floor{\frac{n}{2}}+1}$. However, since the token is now guaranteed to return to both $u_{\floor{\frac{n}{2}}+3}$ and $ u_{\floor{\frac{n}{2}}+1}$ arbitrarily many times we can guarantee that the Explorer will call all path lengths $1,2, \ldots, \lfloor \frac{n}{2}\rfloor-2$ when at $u_{\floor{\frac{n}{2}}+3}$ and $ u_{\floor{\frac{n}{2}}+1}$. 
These lengths are small enough that the Director would be forced to move the token from either $u_{\floor{\frac{n}{2}}+3}$ or $ u_{\floor{\frac{n}{2}}+1}$ to a vertex in $\{u_i\}_{i=3}^n$, and thus in $A$. 
Additionally lengths $\ell \le \frac{n}{2}\rfloor-2$ are small enough that is there is a length $\ell$ $u_iu_{\floor{\frac{n}{2}}+3}$ then $\dist(u_i, u_{\floor{\frac{n}{2}}+3})=\ell$.
This is similarly true for $u_{\floor{\frac{n}{2}}+1}$.
Furthermore, only a path length of $2$ allows the Director move the token from $u_{\floor{\frac{n}{2}}+3}$ to $ u_{\floor{\frac{n}{2}}+1}$ or vice versa. Therefore, by the Explorer calling all path lengths $1,2, \ldots, \lfloor \frac{n}{2}\rfloor-2$ when at $u_{\floor{\frac{n}{2}}+3}$ we have that $|A \setminus \{u_{\floor{\frac{n}{2}}+3}, u_{\floor{\frac{n}{2}}+1}\}| \ge \lfloor\frac{n}{2}\rfloor-3$. Thus 
to keep $|A \setminus \{u_{\floor{\frac{n}{2}}+3}, u_{\floor{\frac{n}{2}}+1}\}|\le \lfloor\frac{n}{2}\rfloor -3$ we now need that every $a \in A$ has $\dist(a, u_{\floor{\frac{n}{2}}+3})$ and $\dist(a, u_{\floor{\frac{n}{2}}+1})$ both at most $\lfloor \frac{n}{2}\rfloor-2$.
This immediately implies that $u_3, u_4, u_{n-1}, u_n$ are not in $A$ as each of these vertices has distance at least $\lfloor\frac{n}{2}\rfloor-1$ to one of $u_{\floor{\frac{n}{2}}+1}$ or $u_{\floor{\frac{n}{2}}+3}$, assuming $n \ge 11$. However, we note that if the token is on a leaf, which we may assume occurs arbitrarily many times, then if the director selects a path length of $\lfloor\frac{n}{2}\rfloor +n-2$ the token must move to one of $u_3, u_4, u_{n-1}, u_n$. In either case we can now see that the Explorer can force the token to visit at least $\lfloor \frac{n}{2}\rfloor$ vertices in $\{u_i\}_{i=3}^n$ bringing the total number of vertices the token must visit to at least $3 \lfloor\frac{n}{2}\rfloor$, as desired. 
    
The even case is quite similar. Again when the token is at $u_2$, the Explorer can choose a path length of $\frac{n}{2}+1$ which now forces the token to either $u_{\frac{n}{2}+1}$ or $u_{\frac{n}{2}+3}$. Now for any vertex $a \in \{u_{\frac{n}{2}},u_{\frac{n}{2}+1},u_{\frac{n}{2}+2},u_{\frac{n}{2}+3}\}$ and $d\in\{1, 2, \ldots, \frac{n}{2}-3\}$, we have every $av_i$ path and $aw_i$ path is of length greater than $\frac{n}{2}-2$, but there is some $u_i$ such that there is an $au_i$ path of length $d$. Furthermore, if we have $d_1, d_2\in \{1, 2, \ldots, \frac{n}{2}-3\}$, then we know that it is impossible to have some $u_i$ such that there is both a $au_i$ path of length $d_1$ and a $au_i$ path of length $d_2$. Again, as the Explorer can force the token to visit both leaves an arbitrarily large number of times they can also force the token to visit $\{u_{\frac{n}{2}},u_{\frac{n}{2}+1},u_{\frac{n}{2}+2},u_{\frac{n}{2}+3}\}$ an arbitrary number of times.
    
As before, it just remains to show that the Director cannot restrict the token to some set of the form $A \cup \{v_i\}_{i=0}^{ \frac{n}{2}- 1} \cup \{w_i\}_{i=0}^{\frac{n}{2} - 1}$ where $A \subseteq \{u_i\}_{i=3}^n$ and $|A|\le \frac{n}{2} -2$. However, we again note that if the token visits either $u_{\frac{n}{2}+1}$ or $u_{\frac{n}{2}+2}$ at least $\frac{n}{2}-2$ times, then, as above, the Explorer, by picking path lengths from $1$ to $\frac{n}{2}-2$, can force the token to visit at least $\frac{n}{2}-1$ vertices in $\{u_i\}_{i=3}^n$.
If not, then again both $u_{\frac{n}{2}}$ and $u_{\frac{n}{2}+3}$ must be visited arbitrarily many times. Therefore, we can again guarantee that the Explorer will call all path lengths $1, \ldots, \frac{n}{2}-3 $ when at both $u_{\frac{n}{2}}$ and $u_{\frac{n}{2}+3}$. Again in order to keep $|A|\le \frac{n}{2} -2$ each vertex $a \in A$ must have $\dist(a,u_{\frac{n}{2}} )$ and $\dist(a,u_{\frac{n}{2}+3})$ at most $ \frac{n}{2}-3$. When $n = 10$ simply the requirement that in this case both $u_{\frac{n}{2}}, u_{\frac{n}{2}+3} \in A$ contradicts this claim as they are distance $3$ apart. When $n \ge 12$ this restriction implies $u_3, u_4, u_5, u_{n-2},u_{n-1},u_n \notin A$ as they are all too far from either $u_{\frac{n}{2}}$ or $u_{\frac{n}{2}+3}$. However, again once the token is on a leaf the Explorer can select a path length of $\frac{3n}{2}-3$ and the token will be forced to a vertex in $\{u_3, u_4, u_5, u_{n-2},u_{n-1},v_n\} $. Thus again the Director will eventually be forced to move the token to at least $\frac{n}{2}-1$ vertices in $\{u_i\}_{i=3}^{n}$ resulting in at least $ \frac{3n}{2}-1$ visited vertices, as desired. 
    
Next, we show the two upper bounds. To do this, we will demonstrate path closed sets of the desired sizes. For convenience in notation, let $V = \{v_i\}_{i=1}^{\floor{\frac{n}{2}}-1}$ and $W= \{w_i\}_{i=1}^{\floor{\frac{n}{2}}-1}$. First we consider the case when $n$ is odd. We claim that $$A=V \cup W \cup \left\{u_i: 1 \le i \le \frac{n+3}{2}\right\}$$ is a path closed set. Note that $|A|= 3\left\lfloor\frac{n}{2}\right\rfloor$, as desired. 
    
We will first consider the case when the token is on some vertex $v_i \in V$ and show that regardless of the path length, $\ell$, chosen by the Explorer the Director can always keep the token in $A$. First, assume that the Explorer selects a path length of $\ell \le i + \left\lfloor\frac{n}{2}\right\rfloor$. Then, we know there exists some vertex $a \in \{v_j: j<i\}\cup W \cup \{u_1, u_2\}$ such that $\dist(a,v_i)=\ell$. Thus, the Director can choose to move the token to $a$ and clearly $a \in A$ as desired. Now consider $\ell \ge i + \left\lfloor\frac{n}{2}\right\rfloor+1$. Note that the path $P=v_iv_{i-1}\ldots v_1u_1u_nu_{n-1}\ldots u_{\frac{n+3}{2}}$ has length $i+\left\lfloor\frac{n}{2}\right\rfloor$. Thus, the Director can choose to follow path $P$ appended with as many additional vertices from $\left\{u_i: 2\le i\le \frac{n+1}{2}\right\} \cup W$ as needed to find a vertex in $a \in A$ with a $v_ia$ path of length $\ell$.

    Now we consider when the token is on some $w_i \in W$. Initially, this is quite similar to the above case. Again, we start by considering when the Explorer selects a path length of $\ell \le i + \left\lfloor\frac{n}{2}\right\rfloor$. Here we know there exists some vertex in $a \in \{w_j: j<i\}\cup V \cup \{u_1, u_2\}$ such that $\dist(a,w_i)=\ell$. However, now we must consider two additional cases. First assume $i + \left\lfloor\frac{n}{2}\right\rfloor+1\le \ell \le i+n-1$. Here we can consider the path $P=w_iw_{i-1}\ldots w_1u_2u_1u_nu_{n-1}\ldots u_{\frac{n+3}{2}}$ which has length $i+\left\lfloor\frac{n}{2}\right\rfloor+1$. Thus, the Director can still choose to follow $P$ appended with a suitable number of vertices from $\left\{u_i: 3\le i\le \frac{n+1}{2}\right\}$ as needed to find a vertex in $a \in A$ with a $w_ia$ path of length $\ell$. Finally, if $\ell \ge i+n$ we note that the path $Q=w_iw_{i-1}\ldots w_1u_2u_3\ldots u_nu_1$ has length $i+n-1$. Now, the Director can choose to follow $Q$ appended with a suitable number of vertices from $V$ to find a vertex in $a \in A$ with a $w_ia$ path of length $\ell$. 

    Now we are only left to consider the case where the token is on some vertex $u_i \in A$. First consider $\ell \le i-1$. Then the path $P = u_iu_{i-1}\ldots u_{i-l}$ has length $\ell$ and ends at $u_{i-\ell} \in A$. If $i\le \ell\le i+ \left\lfloor \frac{n}{2}\right\rfloor -2$, then the Director can follow path $P$ from $u_i$ appended with a suitable number of vertices from $V$ to find a vertex in $a \in A$ with a $u_ia$ path of length $\ell$. 
    Now we consider $i+ \left\lfloor \frac{n}{2}\right\rfloor -1 \le \ell\le n-1$. Note that if $i = \frac{n+3}{2}$, then this case is empty. However, if $1\le i\le \frac{n+1}{2}$, then we note that $Q=u_iu_{i-1}\ldots u_1u_n u_{n-1}\ldots u_{\frac{n+3}{2}}$ is a path of length $i+ \left\lfloor \frac{n}{2}\right\rfloor -1$. Thus from $u_i$, the Director can follow $Q$ and append with a suitable number of vertices from $u_j$ where $i+1 \le j \le \frac{n+1}{2}$ to find a vertex in $a \in A$ with a $u_ia$ path of length $\ell$. Finally, we consider the case where $\ell \ge n$. Here, if $i \neq 2$ the Director will follow the longer of the two $u_iu_2$ paths followed by a suitable number of vertices from $W$. If $i=2$ then the Director will follow the longer of the two $u_2u_1$ paths followed by a suitable number of vertices from $V$.

    Thus, for every possible vertex $a \in A$ which the token is currently on and path length the Explorer could select, the Director can keep the token in $A$. Furthermore, we note that, due to the symmetry of $CF_n$, $$B=V \cup W \cup \left\{u_i: \frac{n+3}{2} \le i \le n \right\} \cup \{u_1,u_2\}$$ is a path closed set of the same size as $A$. Furthermore, each vertex $x \in V(CF_n)$ is an element of either $A$ or $B$. Thus, every vertex is contained in a path closed set with size matching the desired bound proving that when $n$ is odd, regardless of your choice of $x$, we know  $f_p(CF_n,x) = 3\left\lfloor\frac{n}{2}\right\rfloor$.  

    Now we will give a similar construction when $n$ is even.  We claim that $$A=V \cup W \cup \left\{u_i: 1 \le i \le \frac{n}{2}+1\right\}$$ is a path closed set. Note that $|A|= \frac{3n}{2}-1$, as desired. First consider when the token is on a vertex $v_i$ where $i \ge 0$, i.e. we are including the possibility that the token is on $u_1=v_0$. Then if $\ell \le i+ \frac{n}{2}$, then there is a vertex $a \in \{v_j: j\le i\} \cup W \cup \{u_2\}$ such that $\dist(v_i,a)=\ell$. Now assume $\ell \ge i + \frac{n}{2}+1$. Here, the path $P = v_iv_{i-1}\ldots v_0 u_n u_{n-1}\ldots u_{\frac{n}{2}}$ has length $i + \frac{n}{2}+1$. Thus the Director can simply follow $P$ appended with a suitable number of vertices from $\{u_i: \frac{n}{2}-1 \ge i \ge 2\} \cup W$ to find a vertex in $a \in A$ with a $v_ia$ path of length $\ell$.
    
    Now we will consider when the token is on some $w_i$ where $i \ge 0$, similarly including $u_2 = w_0$. If $\ell \le i + \frac{n}{2}$, there is a vertex $a \in \{w_j: j\le i\} \cup V \cup \{u_1\}$ such that $\dist(w_i,a)=\ell$.  Now assume $ i + \frac{n}{2}+1\le \ell \le i+n-1$. Here, the path $P = w_iw_{i-1}\ldots w_0 u_1u_n u_{n-1}\ldots u_{\frac{n}{2}+1}$ has length $i + \frac{n}{2}+1$. Thus, the Director can simply follow $P$ appended with a suitable number of vertices from $\{u_i: \frac{n}{2} \ge i \ge 3\}$ to find a vertex in $a \in A$ with a $w_ia$ path of length $\ell$. Finally, if $\ell \ge i+n$, we must simply consider the path $Q = w_iw_{i-1}\ldots w_1 u_2 u_3u_4\ldots u_{n}u_1$ which has length $i + n-1$. Thus, the Director can simply follow $Q$ appended with a suitable number of vertices from $V$ to find a vertex in $a \in A$ with a $w_ia$ path of length $\ell$.

    The only remaining cases are when the token is on some $u_i$ where $3 \le i \le \frac{n}{2}+1$. Here we start by considering path lengths $\ell \le i-1$. In this case, $\dist(u_i, u_{i-l})=\ell$, and $u_{i-l} \in A$, so the Director can clearly keep the token in $A$. If $i \le \ell \le i+ \frac{n}{2}-2$, we note that the path $P = u_iu_{i-1}\ldots u_1$ has length $i-1$ so the Director can follow $P$ appended with a suitable number of vertices from $V$ to find a vertex in $a \in A$ with a $u_ia$ path of length $\ell$. Next consider when $ i+ \frac{n}{2}-1\le \ell\le n-1$. Note if $i = \frac{n}{2}+1$, this case is empty. Otherwise, the path $Q= u_iu_{i-1}\ldots u_1u_nu_{n-1}\ldots u_{\frac{n}{2}+1}$ has length $i+\frac{n}{2}-1$. Thus the Director can follow $Q$ appended with a suitable number of vertices from $\{u_j: \frac{n}{2}+2 \ge j \ge i+1\}$ to find a vertex in $a \in A$ with a $u_ia$ path of length $\ell$. 
    Finally, consider when $\ell \ge n$. Here, we note that the path $R = u_iu_{i+1}\ldots u_nu_1u_2$ has length $n-i+2$. The Director can follow $R$ appended with a suitable number of vertices from $W$ to find a vertex in $a \in A$ with a $u_ia$ path of length $\ell$. 

     Likewise with the case where $n$ is odd, by the symmetry of $CF_n$, we note that $$B=V \cup W \cup \left\{u_i: \frac{n}{2}+2 \le i \le n \right\} \cup \{u_1,u_2\}$$ is also a path closed set with $|B|=|A|$. Again, each vertex $x \in V(CF_n)$ is an element of either $A$ or $B$. Thus, every vertex is contained in a path closed set with size matching the desired bound proving that when $n$ is even, regardless of your choice of $x$, we know  $f_p(CF_n,x) = \frac{3n}{2}-1$, as desired.  
\end{proof} 

We will now prove Theorem \ref{JF-dist}.
\begin{proof}
We will again begin with the lower bound. For all the cases except for when $n$ is odd and the starting vertex is $u_{\frac{n+3}{2}}$, the lower bound follows from Corollary \ref{ecc}. We must simply show that there are vertices of eccentricity $2\left\lfloor\frac{n}{2}\right\rfloor-1$ in $CF_n$ and that regardless of the starting position of the token the Explorer can always force the token to a vertex of maximum eccentricity. We note that regardless of the parity of $n$, the two leaves of $CF_n$ are always vertices of maximum eccentricity and their eccentricity is $2\left\lfloor\frac{n}{2}\right\rfloor-1$. Additionally, when $n$ is even $u_{\frac{n}{2}+1}$ and $u_{\frac{n}{2}+2}$ also have maximum eccentricity, and when $n$ is odd $u_{\frac{n+3}{2}}, u_{\frac{n+1}{2}}$, and $u_{\frac{n-1}{2}}$ also have maximum eccentricity. Furthermore, if the token is on any vertex $a$ and the Explorer selects a distance of $ecc(a)$ the token must be moved to a vertex of maximum eccentricity. This is sufficient to show that when $n$ is even and $v$ is any starting vertex we know $f_d(CF_n,v) \ge 2\left\lfloor\frac{n}{2}\right\rfloor$. We will show below that in all cases except when $n$ is odd and the starting vertex is $u_{\frac{n+3}{2}}$ this lower bound is, in fact, equality. 

However, when $n$ is odd and the starting vertex is $u_{\frac{n+3}{2}}$ we show that the Explorer can force the token to visit an additional vertex. First note that as in the proof of Theorem \ref{JF-path}, if the token is ever on $v_{\floor{\frac{n}{2}}-1}$ then the Explorer can force the token to visit each $v_i$, where $i \ge 0$ (i.e. including $v_0=u_1$). Similarly, if the token is ever on $w_{\floor{\frac{n}{2}}-1}$, then the Explorer can force the token to visit each $w_i$, where $i \ge 0$ (i.e. including $w_0=u_2$). As a reminder, the Explorer's strategy here is that if the token is on $v_i$ where $1 \le i\le  \floor{\frac{n}{2}}-1$ then the Explorer will select a distance of $d=\min\{\dist(v_i, v_j): v_j\in U\}$ where
$$U = \{v_j: j<i \text{ and } v_j \text{ has not yet been visited}\}.$$ 
The Director has two choices, either they can move the token to $v_j$ or, if possible, move the token ``backwards" to $v_{i+d}$. However, once $i+d> \floor{\frac{n}{2}}-1$, the Director will have no choice but to move the token to a previously unvisited vertex. 
Thus, if the token begins at $u_{\frac{n+3}{2}}$ and at some point visits both $v_{\floor{\frac{n}{2}}-1}$ and $w_{\floor{\frac{n}{2}}-1}$, then we know the Explorer can force the token to at least $2\left\lfloor\frac{n}{2}\right\rfloor+1$ vertices, as desired. Furthermore, since the token starts at $u_{\frac{n+3}{2}}$ the Explorer will first call a distance of $n-2$ which forces the Director to immediately move the token to one of $v_{\floor{\frac{n}{2}}-1}$ or $w_{\floor{\frac{n}{2}}-1}$. 
Thus we will assume, without loss of generality, that the Director plays such that on the first turn they move the token to $v_{\floor{\frac{n}{2}}-1}$ and for the rest of the game they ensure that the token never visits $w_{\floor{\frac{n}{2}}-1}$.
Now we note that for $u_i$ where $3 \le i \le \frac{n+1}{2}$, we have $v_{\floor{\frac{n}{2}}-1}$ being the unique vertex of distance $ecc(u_i)$ from $u_i$. Similarly for $u_i$ where $\frac{n+5}{2}\le i \le n$, we have $w_{\floor{\frac{n}{2}}-1}$ being the unique vertex of distance $ecc(u_i)$ from $u_i$. 
Finally, $u_{\frac{n+3}{2}}$ has exactly two vertices, $w_{\floor{\frac{n}{2}}-1}$ and $v_{\floor{\frac{n}{2}}-1}$, of distance $ecc(u_{\frac{n+3}{2}})$ away. Additionally, for any $w_i$ where $0\le i$, we know there are three vertices $u_{\frac{n+5}{2}}, u_{\frac{n+3}{2}}$, and $v_{\floor{\frac{n}{2}}-1}$ of distance $ecc(w_i)$ from $w_i$. Similarly for any $v_i$ where $i\ge 0$, we know there are three vertices $u_{\frac{n+3}{2}}, u_{\frac{n+1}{2}}$, and $w_{\floor{\frac{n}{2}}-1}$ of distance $ecc(v_i)$ from $v_i$. 
Therefore, if the token begins at $u_{\frac{n+3}{2}}$ and we know the Director will avoid moving the token to $w_{\floor{\frac{n}{2}}-1}$, then the Explorer can force the token to visit $v_{\floor{\frac{n}{2}}-1}$ an arbitrarily large number of times. 
Now, when the token begins at $u_{\frac{n+3}{2}}$, the Explorer can first call a distance of $1$. As the Director does not want to put the token in a position where they can be forced to move it to $w_{\floor{\frac{n}{2}}-1}$, they will move the token to $u_{\frac{n+1}{2}}$. This means the token has now visited two vertices of distance $ecc(v_{\floor{\frac{n}{2}}-1})$ from $v_{\floor{\frac{n}{2}}-1}$. Thus, we know the token must visit at least $ecc(v_{\floor{\frac{n}{2}}-1})+2 = 2\left\lfloor\frac{n}{2}\right\rfloor+1$ vertices, proving $f_d(CF_n,u_{\frac{n+3}{2}}) \ge 2\left\lfloor\frac{n}{2}\right\rfloor+1$.

For the upper bound, we will again show that as long as the starting vertex $v$ is not $u_{\frac{n+3}{2}}$, there is a closed set of size $2\left\lfloor\frac{n}{2}\right\rfloor$ containing $v$. We claim that $A= V \cup \{u_i: 1 \le i \le \left\lfloor\frac{n}{2}\right\rfloor+1\}$ is a closed set. Note that $|A|= 2\left\lfloor\frac{n}{2}\right\rfloor$. We will first consider vertices $v_i \in V$. Here we have that $ecc(v_i)= i+ \left\lfloor \frac{n}{2}\right\rfloor$. Furthermore for any $1 \le d \le i-1$, we know $\dist(v_i, v_{i-d})=d$. If $i \le d \le i+ \left\lfloor \frac{n}{2}\right\rfloor$, then $\dist(v_i, u_{d-i+1})=d$. Thus, since all elements of $V$ are in $A$ and $d-i+1 \le \left\lfloor \frac{n}{2}\right\rfloor+1$, we know the Director can always keep the token in $A$ if it is currently on some $v_i$. If the token is on $u_1$, then we note $ecc(u_1)=\left\lfloor\frac{n}{2} \right \rfloor$ and for any distance $1 \le d \le \left\lfloor\frac{n}{2} \right \rfloor$ we have $\dist(u_1, u_{1+d})=d$. Thus if the token is on $u_1$, the Director can again ensure that it will stay in $A$ regardless of what distance the Explorer selects. Finally, we consider when the vertex is on $u_i$ for $2\le i\le\left\lfloor\frac{n}{2} \right \rfloor+1$. Here, $ecc(u_i)=i+\left\lfloor\frac{n}{2} \right \rfloor-2$. Furthermore, if $1 \le d \le i-1$, then $\dist(u_i, u_{i-d})=d$. Finally, if $i \le d \le i+\left\lfloor\frac{n}{2} \right \rfloor-2$, then $\dist(u_i,v_{d-i+1})=d$. Note that for this range of $d$, we have $1 \le d-i+1 \le \left\lfloor\frac{n}{2} \right \rfloor-1$. Thus, $A$ is a closed set, as desired. 

We note that when $n$ is even, by the symmetry of $CF_n$, we have $B= W \cup \{u_i: \frac{n}{2}+2 \le i \le n\}\cup \{u_1,u_2\}$ is also a closed set of the same size as $A$. Furthermore, each vertex of $CF_n$ lies in either $A$ or $B$. Thus regardless of the starting vertex $v$, we know $f_d(CF_n,v) = 2\left\lfloor\frac{n}{2}\right\rfloor$ for even $n$. When $n$ is odd, the set $B$ we get by considering the same symmetries is $B= W \cup \{u_i: \frac{n+5}{2} \le i \le n\}\cup \{u_1,u_2\}$. Now we note that $u_{\frac{n+3}{2}}$ is in neither $A$ nor $B$. However, we have already shown that the lower bound when the starting vertex is $u_{\frac{n+3}{2}} $ is $2\left\lfloor\frac{n}{2}\right\rfloor+1$. Clearly since we have shown $A$ is always a closed set we also know $C=V \cup \{u_i: 1 \le i \le {\frac{n+3}{2}}\}$ is also a closed set of size $2\left\lfloor\frac{n}{2}\right\rfloor+1$ containing $u_\frac{n+3}{2}$. Thus we have $f_d(CF_n, u_{\frac{n+3}{2}})=2\left\lfloor\frac{n}{2}\right\rfloor+1$ as desired.

\end{proof}
\section{Conclusion and Further Research}

We conclude this paper with a discussion on further research. This paper began the study of $f_p(G,v)$ on a variety of graphs. It would be interesting to work towards some classification of when $f_p(G,v) \ge f_d(G,v)$. For instance, here we showed that for all bipanpositionable graphs with $ecc(G)\ge 3$, we know $f_d(G,v) \ge f_p(G,v)$. Additionally, it would be interesting to investigate if there are infinite graph families where $\frac{f_p(G,v)}{f_d(G,v)}$ grows as the number of vertices of $G$ grows. 

One can also consider the number of steps required to visit $f_p(G,v)$ vertices for various graph families, as was done for the traditional Explorer-Director game in a variety of papers \cite{CLST,HPW,N10,N14,ED}. This then furthers the research into determining the complexity of the problem. Similarly, work into what graph families allow for non-adaptive strategies for the Explorer (i.e. strategies for the Explorer which are independent of the Director's choices) which still achieve $f_p(G,v)$ visited vertices is of interest. Such non-adaptive strategies have already been explored for $f_d(P_n,v)$ in \cite{ED}. Finally, there are clearly many more interesting graph families, including random graphs, for which the values and discrepancy between $f_d(G,v)$ and $f_p(G,v)$ would be of interest.

\nocite{*}
\bibliographystyle{abbrv}
\bibliography{ED2bib}
\end{document}